\documentclass[runningheads]{llncs}

\usepackage{graphicx}
\usepackage{url} 
\usepackage{float}
\usepackage{amsmath,amssymb,bm}
\usepackage{xcolor}
\usepackage{xspace}
\usepackage{hyperref}
\usepackage{xparse}

\newcommand{\ER}{Erd\H{o}s-R\'enyi\xspace}
\newcommand{\fr}{Fr\'echet\xspace}

\DeclareRobustCommand{\qed}{\hfill$\square$}
\DeclareMathOperator{\theargmin}{argmin}

\DeclareMathOperator{\variance}{var}
\DeclareRobustCommand{\e}[1]{{e\big(#1\big)}}
\DeclareRobustCommand{\o}[1]{{\scriptstyle \mathcal{O}} \left(#1\right)}

\newcommand{\eqdef}{\stackrel{\text{\tiny def}}{=}}

\DeclareMathOperator{\R}{\mathbb{R}}

\newcommand{\bA}{\bm{A}}
\newcommand{\bAk}{\bm{A}^{(k)}}

\newcommand{\aij}{a_{ij}}

\newcommand{\akij}{a^{(k)}_{ij}}

\newcommand{\bB}{\bm{B}}

\newcommand{\bP}{\bm{P}}

\newcommand{\cG}{\mathcal{G}}

\newcommand{\cL}{\Lambda}

\newcommand{\cS}{\mathcal{S}}
\newcommand{\Gk}{G^{(k)}}

\newcommand{\gnP}{\mathcal{G}\mspace{-1mu}\big(\mspace{-2mu}n,\bP\mspace{-1mu}\big)}

\DeclareRobustCommand{\ds}[1]{d_{\lambda} \big(#1\big)} 

\DeclareRobustCommand{\var}[1]{\variance\left[#1\right]} 
\DeclareRobustCommand{\argmin}[1]{\underset{#1}{\theargmin}\mspace{4mu}} 

\DeclareRobustCommand{\E}[1]{\mathbb{E} \left [#1 \right]}                                      
\DeclareRobustCommand{\sE}[1]{\widehat{\mathbb{E}}_N\mspace{-2mu}\big [#1 \big]} 

\DeclareRobustCommand{\fm}[1]{\bm{\mu} \big [#1 \big]}                 

\DeclareRobustCommand{\sFE}[1]{\widehat{\bm{\mu}}_N \mspace{-2mu}\big [\mspace{-2mu}#1 \mspace{-2mu}\big]} 
\DeclareRobustCommand{\smd}[1]{\widehat{\bm{m}}_N\big [#1 \big]}    
\DeclareRobustCommand{\md}[1]{\bm{m}\mspace{-2mu} \big [#1 \big]} 

\DeclareRobustCommand{\edg}[1]{\mathcal{E}\big(#1\big)}

\newcommand{\emea}{e_{\widehat{\bm{\mu}}}}       
\newcommand{\emed}{e_{\widehat{\bm{m}}}}        

\newcommand{\edgmd}{\mathcal{E}_{\widehat{\bm{m}}}}     

\newcommand{\ebar}{{\overline{e}}_N}
\newcommand{\stde}{{\sigma_N(e)}}
\newcommand{\vare}{{\sigma_N^2(e)}}

\newcommand{\mA}{\smd{\bA}}       
\newcommand{\mG}{\smd{G}}         

\newcommand{\EA}{\sFE{\bA}}       
\newcommand{\EG}{\sFE{G}}         

\NewDocumentCommand
\F{o m}   
{
  \IfNoValueTF {#2}{\widehat{F}\left(#2\right)}{\widehat{F}_{#1}\mspace{-4mu}\left(#2\right) }
}

\newcommand{\espc}{\sE{\blb(\bA)}} 

\newcommand{\blb}{\bm{\lambda}}

\newcommand{\lAk}{\blb \big (\mspace{-1mu} \bAk\mspace{-2mu} \big)} 
\newcommand{\lAz}{\blb \big (\mspace{-1mu} \bA^{({k_0})}\mspace{-2mu} \big)} 

\newcommand{\spcM}{\blb(\mA)} 
\newcommand{\spcE}{\blb(\sFE{\bA})} 

\newcommand{\bx}{\bm{x}}

\newtheorem{cproof}{Proof of Corollary}

\newtheorem{lproof}{Proof of Lemma}

\begin{document}
\title{On the Number of Edges of the \fr\\ Mean and Median Graphs
  \thanks{This work was supported by the National Science Foundation, CCF/CIF 1815971}}
\author{Daniel Ferguson and Fran\c{c}ois G. Meyer}
\authorrunning{D. Ferguson  and F.G. Meyer}
\institute{Applied Mathematics, University of Colorado at Boulder, Boulder CO 80309
  \url{https://francoismeyer.github.io/}\\ \href{mailto:fmeyer@colorado.edu}{\sf fmeyer@colorado.edu}\\
  {\small \url{https://orcid.org/0000-0002-1529-3796}}
}
\maketitle
\begin{abstract}
The availability of large datasets composed of graphs creates an unprecedented need to invent novel tools in statistical learning for graph-valued random variables. To characterize the average of a sample of graphs, one can compute the sample Frechet mean and median graphs. In this paper, we address the following foundational question: does a mean or median graph inherit the structural properties of the graphs in the sample? An important graph property is the edge density; we establish that edge density is an hereditary property, which can be transmitted from a graph sample to its sample Frechet mean or median graphs, irrespective of the method used to estimate the mean or the median. Because of the prominence of the Frechet mean in graph-valued machine learning, this novel theoretical result has some significant practical consequences.
\end{abstract}
\section{Introduction}
We consider the set $\cG$ formed by all undirected unweighted simple labeled graphs with vertex set $\left\{1, \ldots
,n\right\}$. We equip $\cG$ with a metric $d$ to measure the distance between two graphs.

We characterize the ``average'' of a sample of graphs $\left\{G^{(1)}, \ldots, G^{(N)}\right\}$, which are defined on
the same vertex set $\left\{1, \ldots ,n\right\}$, with the sample \fr mean and median graphs, \cite{frechet47}.
\begin{definition}
  The sample \fr mean graphs are solutions to 
  \begin{equation}
    \EG =  \argmin{G\in \cG} \frac{1}{N}\sum_{k=1}^N d^2(G,\Gk),
    \label{sample-frechet-mean}
  \end{equation}
  and the sample \fr median graphs are solutions to 
  \begin{equation}
    \mG =  \argmin{G\in \cG} \frac{1}{N}\sum_{k=1}^N d(G,\Gk).  
    \label{sample-frechet-median}
  \end{equation}
\end{definition}
Solutions to the minimization problems (\ref{sample-frechet-mean}) and (\ref{sample-frechet-median}) always exist, but
the minimizers need not be unique. All our results are stated in terms of any of the elements in the set of minimizers
of (\ref{sample-frechet-mean}) and (\ref{sample-frechet-median}).

Because the focus of this work is not the computation of the \fr mean or median graphs, but rather a theoretical
analysis of the properties that these graphs inherit from the graph sample, we assume that the graphs in the sample are
defined on the same vertex set.

The vital role played by the \fr mean as a location parameter \cite{jain12,jain16a,kolaczyk20}, is exemplified in the
works of \cite{banks98,lunagomez20}, who have created novel families of random graphs by generating random perturbations
around a given \fr mean graph.
\subsection{Our main contributions}
We consider a set of $N$ unweighted simple labeled graphs, $\big\{G^{(1)},\ldots, G^{(N)} \big\}$, with vertex set
$\left\{1, \ldots ,n\right\}$.  In this paper, we address the following foundational question: does a mean or median graph
inherit the structural properties of the graphs in the sample?
Specifically, we establish that edge density is an hereditary property, which can be transmitted from a graph sample to
its sample \fr mean or median.

Because sparse graphs provide prototypical models for real networks, our theoretical analysis is significant since it
provides a guarantee that this structural property is preserved when computing a sample mean or median. In a similar
vein, the authors in \cite{han16} construct a sparse median graph, which provides a more interpretable summary, from a
set of graphs that are not necessarily sparse.

Our work answers the question raised by the author in \cite{ginestet17}: ``does the average of two sparse
networks/matrices need to be sparse?'' Specifically, we prove the following result: the number of edges of the \fr mean
or median graphs of a set of graphs is bounded by the sample mean number of edges of the graphs in the sample. We prove
this result for the graph Hamming distance, and the spectral adjacency pseudometric, using different arguments.

\section{Preliminary and Notations}
We denote by $\cS$ the set of $n \times n$ adjacency matrices of graphs in $\cG$,
\begin{equation}
  \cS = \left \{
  \bA \in \{0,1\}^{n \times n}; \text{where} \; a_{ij} = a_{ji},\text{and}  \; a_{i,i} = 0; \; 1 \leq i < j \leq n
  \right\}.
  \label{adjacency_matrices}
\end{equation}
For a graph $G \in \cG$, we denote by $\bA$ its adjacency matrix, and by $\e{\bA}$ the number of edges -- or {\em volume} -- of $G$, 
\begin{equation}
  \e{\bA} = \sum_{1 \leq i < j \leq n} \mspace{-16mu} \aij.
\end{equation}
We denote by $\blb (\bA)= \begin{bmatrix} \lambda_1(\bA)& \cdots& \lambda_n(\bA) \end{bmatrix}$, the vector of
eigenvalues of $\bA$, with the convention that $\lambda_1 (\bA) \ge \ldots \ge \lambda_n(\bA)$.
\subsection{Distances between graphs}
In this work, we consider two metrics: the Hamming distance and the spectral adjacency pseudometric. We briefly recall the
definitions of these.
\begin{definition}
  Let $G,G^\prime \in \cG$ be two unweighted graphs with known vertex correspondence and with adjacency matrix $\bA$ and
  $\bA^\prime$ respectively.  We define the Hamming distance between $G$ and $G^\prime$ as
  \begin{equation}
    d_H(\bA,\bA^\prime) \eqdef \sum_{1 \le i< j \le n} \lvert a_{ij} - a^\prime_{ij}\rvert =  \e{\bA} +  \e{\bB} -
    2 \mspace{-12mu} \sum_{1 \le i < j \le n} \mspace{-12mu} a_{ij}b_{ij}. \label{Hamming_edges}
  \end{equation}
\end{definition}

\noindent The Hamming distance is very sensitive to fine scale fluctuations of the graph connectivity. In contrast, a
metric based on the eigenvalues of the adjacency matrix can quantify configurational changes that occur on a graph at
many more scales \cite{Donnat2018,wills20}.
\begin{definition}
  Let $G,G^\prime \in \cG$ with adjacency matrix $\bA$ and $\bA^\prime$ respectively.  We define the adjacency spectral
  pseudometric as the $\ell_2$ norm between the vectors of eigenvalues $\blb (\bA)$ and $\blb (\bA^\prime)$ of $\bA$ and
  $\bA^\prime$ respectively,
  \begin{align} 
    \ds{\bA,\bA^\prime} = || \blb (\bA) - \blb (\bA^\prime)||_2. \label{distance}
  \end{align}
\end{definition}
The pseudometric $d_{\lambda}$ satisfies the symmetry and triangle inequality axioms, but not the identity axiom. Instead, $d_{\lambda}$
satisfies the reflexivity axiom, $\forall G \in \cG$, $d_{\lambda}(G,G) = 0$. We note that the adjacency spectral pseudometric does not
require node correspondence.
\section{Main Results}
In the following, we consider a set of $N$ unweighted simple labeled graphs, $\big\{G^{(1)},\ldots, G^{(N)} \big\}$, with
vertex set $\left\{1, \ldots ,n\right\}$. We denote by $\bAk$ the adjacency matrix of graph $\Gk$. We equip the set $\cG$
of all unweighted simple graphs on $n$ nodes with a pseudometric, or a metric, $d$. The \fr mean and median graphs encode
two notions of centrality (\ref{sample-frechet-mean}) and (\ref{sample-frechet-median}) that minimise the following
dispersion function, also called the \fr function.
\begin{definition}
  We denote by $\F[q]{\bA}$ the sample \fr function associated with a sample \fr median ($q=1$) or mean ($q=2$),
  \begin{equation}
    \F[q]{\bA} = \frac{1}{N} \sum_{k=1}^N d^q(\bA,\bAk).
    \label{sampleFrechetFunction}
  \end{equation}
\end{definition}
To quantify the connectivity of the graph sample, $\big\{G^{(1)},\ldots, G^{(N)} \big\}$, we define the sample mean and
variance of the number of edges.
\begin{definition}
  The sample mean and variance of the number of edges are defined by
  \begin{equation}
    \ebar = \frac{1}{N} \sum_{k=1}^N \e{\bAk}, \quad \text{and} \quad
    \vare = \frac{1}{N}\sum_{k=1}^N \big[\e{\bAk}\big]^2 - \big[\ebar]^2.\label{ebar}
  \end{equation}
\end{definition}
We now turn our attention to the main problem.  We consider the following question: if the graphs $G^{(1)}, \ldots, G^{(N)}$ all
have a similar edge density, can one determine the edge density of the sample \fr mean or median graphs? and does that number of
edges depend on the choice of metric $d$ in (\ref{sample-frechet-mean}) and (\ref{sample-frechet-median})? We answer both
questions in the following theorem.
\begin{theorem}
  \label{theorem1}
  Let $\big\{G^{(1)},\ldots, G^{(N)} \big\}$ be a sample of unweighted simple labeled graphs with vertex set $\left\{1,
  \ldots ,n\right\}$.  Let $\EA$ be the adjacency matrix of a sample \fr mean graph, and let $\mA$ be the adjacency
  matrix of a sample \fr median graph. Let $\emea$ and $\emed$ be the number of edges of $\EA$ and $\mA$ respectively.\\
  
  \noindent If the \fr mean and median graphs are computed using the Hamming distance, then
  \begin{equation}
    \emea <  2  \ebar + \frac{\stde}{\sqrt{2}} , 
    \quad \text{and} \quad
    \emed <  2  \; \ebar, \label{bound_hamming}
  \end{equation}
  and if the \fr mean and median graphs are computed using the adjacency spectral pseudometric, then
  \begin{equation}
    \emea <  9  \; \ebar, 
    \quad \text{and} \quad
    \emed <  9  \ebar. \label{bound_spectral}
  \end{equation}
\end{theorem}
\begin{proof}
  The proof is a direct consequence of lemmata \ref{lemma5} and \ref{frechet-edge-spectral}.
\end{proof}
\begin{remark}
  When the graph $\Gk$ are sampled from the inhomogeneous \ER random graph probability space $\gnP$ \cite{bollobas07},
  and if the distance on $\cG$ is the Hamming distance, then $\EA = \mA$ with high probability \cite{meyer22b}. In this
  case, a tight bound on $\emea$ or $\emed$ in (\ref{bound_hamming}) is $2\ebar$, which -- unlike (\ref{bound_hamming})
  -- does not involve $\stde$.

  The fact that we overestimate the bound on $\emea$ by the addition of the term $\stde/\sqrt{2}$ comes from our
  technique of proof, which relies on an estimate of the \fr function. As explained in Remark~\ref{remark-lemma4}, our
  estimate of the \fr function is almost tight; it does include the term $\stde$, as it should.
\end{remark}
Finally, the following corollary answers the question raised by the author in \cite{ginestet17}: ``does the average of
two sparse networks/matrices need to be sparse?''
\begin{corollary}
  Let $\big\{G^{(1)},\ldots, G^{(N)} \big\}$ be a sample of unweighted simple labeled graphs with vertex set $\left\{1,
  \ldots ,n\right\}$. We assume that the number of edges of each $\Gk$ satisfies
  \begin{equation}
    \e{\bAk}= \o{n^2},\;\text{but}\; \e{\bAk}  = \omega ( n). \label{sparsity}
  \end{equation}
  Then the sample \fr mean and median graphs -- computed according to either the Hamming distance or the adjacency
  spectral pseudometric -- are sparse, as defined by (\ref{sparsity}).
\end{corollary}
\begin{cproof}
  The corollary is a direct consequence of theorem~\ref{theorem1}.
\end{cproof}
\section{Proofs of the main result}
We give in the following the proof of theorem~\ref{theorem1}. The key observation is that it is relatively easy to
derive tight bounds on the number of edges of the sample \fr median graph. Inspired by the results in \cite{meyer22b}
that show that for large classes of random graphs the sample \fr median and mean graphs are identical, we prove that the
bounds derived for the \fr median graphs also hold for the \fr mean graphs.

Our analysis begins in Subsection \ref{median-hamming} with the sample median graphs computed using the Hamming distance,
we then move to the sample mean graphs in Subsection~\ref{mean-hamming}. In Subsections~\ref{mean-spectral}
and~\ref{median-spectral}, we extend these results to the sample mean and median graphs computed with the adjacency
spectral pseudometric.

When possible, we use the probability space $\gnP$ of inhomogeneous \ER random graphs \cite{bollobas07},
equipped with the Hamming distance to test the tightness of our results \cite{meyer22b}.
\subsection{The median graphs computed using the Hamming Distance
  \label{median-hamming}
}
The Hamming distance, by nature, promotes sparsity \cite{Donnat2018,wills20}, and we therefore expect that the volumes
of the sample \fr mean and median graphs computed with this distance be similar to the sample mean number of edges.

When the distance is the Hamming distance, the sample \fr median graphs can in fact be characterized analytically.
\begin{lemma}
  The adjacency matrix $\mA$ of a sample median graph $\mG$ is given by the majority rule,
  \begin{equation}
    \Big[\mA\Big]_{ij} =
    \begin{cases}
      0&  \text{if} \; \sum_{k=1}^N \akij < N/2,\\
      1 & \text{otherwise.}
    \end{cases}
    \quad \forall i,j \in \left\{1,\ldots,n\right\}.
    \label{majority-rule}
  \end{equation}
\end{lemma}
\begin{lproof}
  The result is classic and we omit the proof, which can be found for instance in \cite{devroye13}.
\end{lproof}
In the following lemma, we derive an upper bound on the number of edges of a \fr median graph, $\emed$.
\begin{lemma}
  \label{lemma2}
  Let $\ebar$ be the sample mean number of edges, given by (\ref{ebar}). Then the number of edges of a \fr median graph $\mG$ is
  bounded by
  \begin{equation}
    \emed  \leq 2 \ebar.
    \label{bound-mbar}
  \end{equation}
\end{lemma}
\begin{remark}
  \label{the_bound_is_tight}
  The bound (\ref{bound-mbar}) is tight for large $N$. Indeed, consider a sample of $2N$ graphs, where
  \begin{equation}
    G^{(k)}=
    \begin{cases}
      \text{the complete graph}\; K_n & \text{if}\quad 1 \leq k \leq N+1,\\
      \text{the empty graph} & \text{if} \quad N+2 \leq k \leq 2N.
    \end{cases}
  \end{equation}
  A \fr median graph $\mA$, given by the majority rule (\ref{majority-rule}) is $K_n$, and thus $\emed = n(n-1)/2$. On
  the other hand, the sample mean number of edges is $\ebar = \emed/2 + \emed/(2N).$ As the sample size $N$ goes to
  infinity, we have
  \begin{equation}
    \lim_{N\longrightarrow \infty} \emed =  2\ebar,
  \end{equation}
  which proves that the bound (\ref{bound-mbar}) is asymptotically tight.
\end{remark}
\begin{lproof}
  \noindent Let $\edgmd = \left\{(i,j),\; i < j, \; [\mA]_{ij} = 1\right\}$ be the set of edges of $\mG$. We have 
  $\lvert \edgmd \rvert = \emed$. Now,
  \begin{equation}
    \sum_{k=1}^N \e{\bAk}
    =  \mspace{-16mu} \sum_{1 \le i < j \le n}  \sum_{k=1}^N  \akij
    = \mspace{-8mu} \sum_{i,j \in \edgmd}  \sum_{k=1}^N  \akij
    +
    \mspace{-8mu}\sum_{i,j \in \edgmd^c}  \sum_{k=1}^N  \akij .
  \end{equation}
  Neglecting the edges $(i,j)$ not in $\edgmd$, we have
  \begin{equation*}
    \sum_{k=1}^N \e{\bAk}
    \ge \sum_{i,j \in \edgmd}  \sum_{k=1}^N  \akij
    > \mspace{-8mu}
    \sum_{i,j \in \edgmd} \frac{N}{2}
    =  
    \frac{N}{2} \emed,
  \end{equation*}
  whence we conclude
  \begin{equation}
    \emed \le \frac{2}{N} \sum_{k=1}^N \e{\bAk} = 2 \ebar. 
  \end{equation}
  \qed
\end{lproof}
\subsection{The mean graphs computed using the Hamming Distance
  \label{mean-hamming}
}
First, we recall the following lower bound on the Hamming distance.
\begin{lemma}
  \label{lower-bound-hamming}
  Let $\bA$ and $\bB$ be the adjacency matrices of two unweighted graphs with number of edges  $\e{\bA}$ and $\e{\bB}$
  respectively. Then
  \begin{equation}
    \big \lvert \e{\bA} - \e{\bB} \big \rvert \leq d_H(\bA,\bB).
  \end{equation}
\end{lemma}
\begin{lproof}
  The proof is elementary and is skipped.
\end{lproof}
\noindent Next, we derive an upper bound on the deviation of the volume of a \fr mean, $\emea$, away from the sample average
volume, $\ebar$, given by (\ref{ebar}).
\begin{lemma}
  \label{lemma3}
  Let $\EA$ be the adjacency matrix of a sample \fr mean computed using the Hamming distance, with
  $\emea$ edges.  Let $\ebar$ be the sample mean number of edges. Then
  \begin{equation}
    \bigg [ \emea - \ebar\bigg ]^2
    \mspace{-8mu}
    < \mspace{-4mu}
    \frac{1}{N}
    \mspace{-4mu}
    \sum_{k=1}^N d^2_H (\EA, \bAk)
    =
    \F[2]{\EA}.
    \label{lower-bound-frechet}
  \end{equation}
\end{lemma}
\begin{remark}
  \label{F2-IER}
  This bound is not tight. We consider again the probability space of inhomogeneous \ER random graphs equipped with the
  Hamming distance. In that case, one can show that the population \fr mean and median coincide \cite{meyer22b}, and the
  adjacency matrix of the population \fr mean graph, $\fm{\bA}$, is given by the majority rule,
  \begin{equation}
    \Big[\fm{\bA}\Big]_{ij} =
    \begin{cases}
      1 & \text{if} \quad p_{ij} > 1/2,\\
      0 & \text{otherwise.}
    \end{cases}
  \end{equation}
Also, the population \fr function, $F_2$, evaluated at $\fm{\bA}$ is given by \cite{meyer22b} 
  \begin{equation}
    F_2(\fm{\bA}) = 
    \Big[\mspace{-12mu}
      \sum_{1\le i < j \le n} \mspace{-12mu} p_{ij}
      - \mspace{-24mu}
      \sum_{(i,j)\in\edg{\fm{\bA}}} \mspace{-24mu} (2 p_{ij} -1)
      \Big]^2
    +
    \mspace{-16mu}
    \sum_{1\le i < j \le n} \mspace{-12mu} p_{ij} (1-p_{ij}),
    \label{frechet-population}
  \end{equation}
  where $\edg{\fm{\bA}}$ is the set of edges of the population \fr mean, $\fm{\bA}$. We claim that the lower bound on
  $\F[2]{\EA}$ in (\ref{lower-bound-frechet}),
\begin{equation}
  \big [ \ebar - \emea \big ]^2, \label{F2-hat-term1}
\end{equation}
can be identified with the first term of $F_2(\fm{\bA})$ in (\ref{frechet-population}),
\begin{equation}
  \Big[ \mspace{-12mu}\sum_{1\le i < j \le n} p_{ij}
    -
    \mspace{-16mu}
    \sum_{(i,j)\in\edg{\fm{\bA}}} \mspace{-16mu}(2 p_{ij} -1) \Big]^2.
  \label{F2-term1}
\end{equation}
Indeed, the first sum inside (\ref{F2-term1}) is the population mean number of edges, $\E{e}$, which matches the sample
mean $\ebar$ in (\ref{F2-hat-term1}).  Also, the second sum in (\ref{F2-term1}) is bounded by $\e{\fm{\bA}}$, the
number of edges of the population \fr mean,
\begin{equation}
  0 < \mspace{-24mu} \sum_{(i,j)\in\edg{\fm{\bA}}} \mspace{-24mu} (2 p_{ij} -1) < \mspace{-24mu} \sum_{(i,j)\in\edg{\fm{\bA}}}
  \mspace{-24mu} 1 = \e{\fm{\bA}}. 
\end{equation}
The number of edges $\e{\fm{\bA}}$ matches the sample estimate, $\emea$, in (\ref{F2-hat-term1}). In summary, the first
term (\ref{F2-term1}) of the population \fr function (\ref{frechet-population}) matches the
corresponding sample estimate (\ref{F2-hat-term1}).

However, the second term, $\sum_{1\le i < j \le n} p_{ij} (1-p_{ij})$ in (\ref{frechet-population}), which accounts for
the variance of the $n(n-1)/2$ independent Bernoulli edges, is not present in the lower bound on in
$F_2[\fm{\bA}]$ given by (\ref{lower-bound-frechet}), confirming that the lower bound in (\ref{lower-bound-frechet}) is
missing a variance term, and is therefore not tight.
\end{remark}

\begin{lproof}
  Because of lemma~\ref{lower-bound-hamming}, we have
  \begin{equation}
    \big  \lvert  \e{\bAk} - \emea \big \rvert^2 \le d^2_H(\EA,\bAk). \label{d2H}
  \end{equation}    
  Now, the function
  \begin{equation}
    x \longmapsto \big (\emea - x\big)^2
  \end{equation}
  is strictly convex so,
  \begin{equation}
    \big \lvert  \ebar  - \emea \big \rvert^2 \mspace{-12mu}
    =  \Bigg \lvert  \frac{1}{N} \mspace{-4mu} \sum_{k=1}^N \e{\bAk} - \emea \Bigg \rvert^2 
    <  \frac{1}{N}    \sum_{k=1}^N \Big \lvert \e{\bAk} - \emea \Big \rvert^2,\label{strictly-convex}
  \end{equation}
  and substituting (\ref{d2H}) for each $k$ in (\ref{strictly-convex}), we get the advertised result. \qed 
\end{lproof}
Finally, we compute an upper bound on the \fr function evaluated at a sample \fr median graph, $\F[2]{\mA}$.
\begin{lemma}
  \label{lemma4}
Let $\ebar$ and $\vare$ be the sample mean and variance of the number of edges (see (\ref{ebar})).  Then the \fr function
$\F[2]{\mA}$ evaluated at a \fr median graph is bounded by
  \begin{equation}
    \F[2]{\mA}
    \le 
    2\big[\ebar\big]^2 +\vare.
    \label{frechet-of-median}
  \end{equation}
\end{lemma}
\begin{remark}
  \label{remark-lemma4}
  As explained in Remark~\ref{F2-IER}, when the graphs $\Gk$ are sampled from $\gnP$, then the population \fr mean and
  median graphs coincide, $\fm{G} = \md{G}$. Also, the population \fr function $F_2(\md{\bA})$ evaluated at a population
  \fr median graph is given by
  \begin{equation}
    F_2\big[\md{\bA}\big] = 
    \Big[\mspace{-12mu}
      \sum_{1\le i < j \le n} \mspace{-12mu} p_{ij}
      - \mspace{-24mu}
      \sum_{(i,j)\in\edg{\md{\bA}}} \mspace{-24mu} (2 p_{ij} -1)
      \Big]^2
    +
    \mspace{-16mu}
    \sum_{1\le i < j \le n} \mspace{-12mu} p_{ij} (1-p_{ij}),
  \end{equation}
  where the term $\sum_{(i,j)\in\edg{\md{\bA}}} (2 p_{ij} -1)$ is always positive (since the median graphs are
  constructed using the majority rule (\ref{majority-rule})). Therefore, we have
  \begin{equation}
    F_2\big [\md{\bA}\big] \le 
    \Big[\mspace{-12mu}
      \sum_{1\le i < j \le n} \mspace{-12mu} p_{ij}
      \Big]^2
    +
    \mspace{-16mu}
    \sum_{1\le i < j \le n} \mspace{-12mu} p_{ij} (1-p_{ij}).
  \end{equation}
  The term $\sum_{1\le i < j \le n}p_{ij}$ is the expectation of the number of
  edges, whereas $\sum_{1\le i < j \le n} p_{ij} (1-p_{ij})$ is the
  variance of the number of edges. In summary, we have the following bound on the population \fr function,
  \begin{equation}
    F_2(\md{\bA}) \leq \big[\E{e}\big]^2 + \var{e}, \label{F2-population-bound}
  \end{equation}
  where $e$ denotes the number of edges in graphs sampled from $\gnP$. If we replace $\E{e}$ and $\var{e}$ by their
  respective sample estimates, $\ebar$ and $\vare$, then the bound (\ref{frechet-of-median}) is only slightly worse (by
  a factor $2$ in front of $\ebar$) than the population bound, (\ref{F2-population-bound}). Interestingly, the variance
  of the number of edges is present in both expressions.
\end{remark}
\begin{lproof}
  \noindent From (\ref{Hamming_edges}), one can derive the following expression for the Hamming
  distance from a \fr median graph $\mG$ to a graph $\Gk$,
  \begin{equation}
    d_H (\mA, \bAk ) 
    =  \emed +  \e{\bAk} -  2 \mspace{-16mu}\sum_{(i,j)\in \edgmd} \akij, \label{hamming-explicit}
  \end{equation}
  where we recall that $\edgmd = \left\{(i,j) , \; i < j, \; \Big[\mA \Big]_{ij} = 1\right\}$ is the set of edges of
  $\mG$. Taking the square of the Hamming distance given by (\ref{hamming-explicit}), and summing over all the graphs, yields
  \begin{align*}
    \F[2]{\mA} 
    =  \frac{1}{N} \sum_{k=1}^N
    \Bigg \{& \bigg [\emed +  \e{\bAk} \bigg]^2
    + 4\bigg[\sum_{(i,j)\in \edgmd} \mspace{-8mu} \akij \bigg]^2\notag \\
    & - 4 \big (\emed +  \e{\bAk}\big) \mspace{-4mu} \Big[ \mspace{-8mu}\sum_{(i,j)\in \edgmd} \mspace{-8mu} \akij \Big ] 
    \Bigg \}. \notag
  \end{align*}
  Expanding all the terms, and using the definition of $\vare$ and $\ebar$ in (\ref{ebar}), we get
  \begin{align}
    \F[2]{\mA} 
    = &  \big[\emed\big]^2  + 2 \emed \; \ebar + \vare + \big[ \ebar \big]^2
    +  \frac{4}{N}\sum_{k=1}^N \Bigg[ \sum_{(i,j)\in \edgmd} \akij \Bigg]^2        \notag\\
    & - \frac{4}{N}\sum_{k=1}^N \e{\bAk}  \Big [\mspace{-16mu} \sum_{(i,j)\in \edgmd} \mspace{-12mu} \akij \Big ]
    - 4  \emed  \Big [ \mspace{-6mu} \sum_{(i,j)\in \edgmd}  \frac{1}{N}\sum_{k=1}^N \akij \Big ]     \notag\\
    = & \big[\emed +\ebar \big]^2  + \vare 
    + 4 \frac{1}{N}\sum_{k=1}^N \Bigg[ \sum_{(i,j)\in \edgmd} \akij \Bigg]^2\notag \\
    & - \frac{4}{N}\sum_{k=1}^N \e{\bAk}  \Big [\mspace{-16mu} \sum_{(i,j)\in \edgmd} \mspace{-12mu} \akij \Big ] 
    - 4  \emed  \Big [ \mspace{-6mu} \sum_{(i,j)\in \edgmd}  \frac{1}{N}\sum_{k=1}^N \akij \Big ]. \label{line43}
  \end{align}
  Now, because of the definition of the median graphs (\ref{majority-rule}), we have the following upper bound
  \begin{equation}
    - 4  \emed  \Big [\mspace{-4mu} \sum_{(i,j)\in \edgmd} \frac{1}{N} \sum_{k=1}^N \akij \Big ]
    \leq
    - 2 \big[\emed\big]^2. \label{for_line43}
  \end{equation} 
  Because $\e{\bAk} \ge \sum_{(i,j) \in \edgmd} \akij$, we get the following upper bound,
  \begin{equation}
    - 4 \sum_{k=1}^N \e{\bAk} \mspace{-8mu} \sum_{(i,j) \in \edgmd}\akij 
    \le
    -4 \sum_{k=1}^{N} \mspace{8mu} \Big[\mspace{-8mu} \sum_{(i,j) \in \edgmd} \akij\Big]^2. \label{for_line42}
  \end{equation}
  Finally, after substituting (\ref{for_line43}) and (\ref{for_line42}) into (\ref{line43}), we get the bound announced
  in the lemma,
  \begin{align}
    \F[2]{\mA} 
    \le & \big[\emed +\ebar \big]^2  - 2 \big[\emed\big]^2 + \vare
    = - \big[\emed  - \ebar \big]^2 + 2 \big[\ebar\big]^2 + \vare \notag\\ 
    \le &  2 \big[\ebar\big]^2 + \vare. \tag*{\qed}
  \end{align}  
\end{lproof}
\subsection{The number of edges of $\mG$ and $\EG$ when $d = d_H$
  \label{frechet-hamming}
}
The following lemma provides the bounds given by Theorem~\ref{theorem1} when $d$ is the Hamming distance.
\begin{lemma}
  \label{lemma5}
  Let $\big\{ G^{(1)}, \ldots, G^{(N)} \big\}$ be a sample of unweighted simple labeled graphs with vertex set
  $\left\{1, \ldots ,n\right\}$. Let $\EA$ be the adjacency matrix of a sample \fr mean graph, and $\mA$ be the
  adjacency matrix of a sample \fr median graph, computed according to the Hamming distance. Then
  \begin{equation}
    \e{\EA} < 2 \ebar + \frac{\stde}{\sqrt{2}}, \quad
    \text{and} \quad
    \e{\mA} \leq 2 \ebar.
    \label{bound-lemma5}
  \end{equation}
\end{lemma}
\begin{lproof}
  The bound on $\e{\mA}$ is a straightforward consequence of lemma~\ref{lemma3}. Indeed,
  (\ref{bound-mbar}) and (\ref{ebar}) yield the bound in (\ref{bound-lemma5}),
  \begin{equation*}
    \e{\mA}  \leq \frac{2}{N} \sum_{k=1}^N \e{\bAk} \leq 2 \ebar.
  \end{equation*}
  \noindent We now move to $\e{\EA}$. We use $\mA$ to derive an upper bound on the \fr function computed at $\EA$. By
  definition of the sample \fr mean graphs, we have
  \begin{equation}
    \frac{1}{N}\sum_{k=1}^N d^2_H(\EA,\bAk )
    \le
    \frac{1}{N}\sum_{k=1}^N d^2_H(\mA,\bAk ).
    \label{mean_beats_median}
  \end{equation}
  Using (\ref{lower-bound-frechet}) as a lower bound and (\ref{frechet-of-median}) as an upper bound in
  (\ref{mean_beats_median}), we get
  \begin{equation*}
    \bigg [\emea  - \ebar\bigg]^2  <  2 \big[\ebar\big]^2 + \vare,
  \end{equation*}
  and thus
  \begin{equation}
    \big \lvert \emea  - \ebar \big \rvert
    \le \sqrt{2 \big[\ebar\big]^2 + \vare}
    \le \frac{1}{\sqrt{2}}\bigg\{\sqrt{2} \ebar + \stde \bigg\} = \ebar + \frac{\stde}{\sqrt{2}},
  \end{equation}
  from which we get the advertised bound on $\emea$. \qed   
\end{lproof}
\subsection{The mean graphs computed using the adjacency spectral pseudometric
  \label{mean-spectral}}
The technical difficulty in defining the sample \fr mean and median graphs according to the adjacency spectral pseudometric
stems from the fact that the sample \fr function, $\F[q]{\bA}$, is defined in the spectral domain, but the domain over which the
optimization takes place is the matrix domain. This leads to the definition of the set, $\cL$, of real spectra that are
realizable by adjacency matrices of unweighted graphs (elements of $\cS$, defined by (\ref{adjacency_matrices}))
\cite{johnson18},
\begin{equation}
  \cL = \left \{
  \blb (\bA) = \begin{bmatrix} \lambda_1(\bA)& \cdots& \lambda_n(\bA) \end{bmatrix}; \text{where} \bA \in \cS
  \right\}.
  \label{realisable}
\end{equation}
Let $\left\{G^{(1)}, \ldots, G^{(N)}\right\}$ be a sample of unweighted simple labeled graphs with vertex set $\left\{1,
\ldots ,n\right\}$. Let $\bAk$ be the adjacency matrix of graph $\Gk$, and let $ \blb(\bAk)$ be the spectrum of
$\bAk$. The adjacency matrix, $\EA$, of a sample \fr mean graph computed according to the adjacency spectral
pseudometric, has a vector of eigenvalues, $\spcE \in \cL$, that satisfies
\begin{equation}
  \spcE 
  = \argmin{\blb  \in \cL }\sum_{k=1}^N ||\blb - \blb(\bAk)||^2.
  \label{def-mean-spectral}
\end{equation}
Similarly, the adjacency matrix, $\mA$, of a sample \fr median computed according to the adjacency
spectral pseudometric, has a vector of eigenvalues, $\spcM \in \cL$, that satisfies
\begin{equation}
  \spcM = \argmin{\blb  \in \cL }\sum_{k=1}^N ||\blb - \blb(\bAk)||.
  \label{def-median-spectral}
\end{equation}
We recall the following result that expresses the number of edges as a function of the $\ell^2$ norm of the
spectrum of the adjacency matrix.
\begin{lemma}
  Let $G \in \cG$ with adjacency matrix $\bA$. Let $\lambda_1 (\bA) \ge \ldots \ge  \lambda_n(\bA)$ be the eigenvalues of $\bA$. Then 
  \begin{equation}
    2 \e{\bA} =     \sum_{i=1}^n \lambda_i^2(\bA) = \| \blb(\bA) \|_2^2. \label{edges-eigenval}
  \end{equation}
\end{lemma}
\begin{lproof}
  The result is classic; see for instance \cite{bapat10,vanmieghen10}.
\end{lproof}
We derive the following lower bound on the sample mean number of edges.
\begin{lemma}
  \label{mean-spectrum-mean-edges}
  Let $\espc  = \frac{1}{N} \sum_{k=1}^N   \blb (\bAk)$ be the sample mean spectrum.  Then
  \begin{equation}
    \frac{1}{2} \Big \| \espc  \Big \|^2 \le \ebar,
  \end{equation}
  where $\ebar$ is the sample mean number of edges, given by (\ref{ebar}).
\end{lemma}
\begin{lproof}
  The result is a straightforward consequence of the convexity of the norm combined with (\ref{edges-eigenval}).
\end{lproof}
If $\cL$ were to be a convex set, then the spectrum of a sample \fr mean graph would simply be the sample mean
spectrum, which would minimize (\ref{def-mean-spectral}). Unfortunately, $\cL$ is not convex \cite{knudsen2001}. We can
nevertheless relate the spectrum of a sample \fr mean graph, $\spcE$, to the mean spectrum $\espc$. We take a short
detour to build some intuition about the geometric position of the spectrum of $\EA$ with respect to $\blb(\bA^{(1)}),
\ldots, \blb(\bA^{(N)})$. 
\subsubsection{Warm-up: The Sample Mean Spectrum.}
\noindent Let $\left\{G^{(1)}, \ldots, G^{(N)}\right\}$ be a sample of unweighted simple labeled graphs with vertex set
$\left\{1, \ldots ,n\right\}$. Let $\bAk$ be the adjacency matrix of graph $\Gk$, and let $ \blb(\bAk)$ be the spectrum
of $\bAk$.
\begin{lemma}
  \label{centerofmass}
  Let $\espc$  be the sample mean spectrum. Then $\exists \; k_0 \in \{1,\ldots,N\}$ such that
  \begin{equation}
    \| \blb (\bA^{(k_0)})\|  \le \|  \espc \|.
  \end{equation}
\end{lemma}
\begin{lproof}
  A proof by contradiction is elementary.
\end{lproof}
Using the characterization of a sample \fr mean graph, $\EA$, given by (\ref{def-mean-spectral}), we can extend the above
lemma to $\spcE$, and derive the following result.
\begin{lemma}
  \label{spectral-mean}
  Let $\spcE$ be the spectrum of a sample \fr mean graph. Let $\ebar$ be the sample mean number of edges of the graphs
  $G^{(1)},\ldots, G^{(N)}$.  Then
  \begin{equation}
    \| \spcE \|   \leq  3 \sqrt{2 \ebar}.
  \end{equation}
\end{lemma}
\begin{lproof}
  Because of lemma~\ref{centerofmass},
  \begin{equation}
    \exists \; k_0 \in \{1,\ldots, N\},\;   \| \blb (\bA^{(k_0)})\| \le \|  \espc \|.
    \label{mean-convex}
  \end{equation}
  Now, because of lemma~\ref{mean-spectrum-mean-edges},(\ref{mean-convex}) implies that
  \begin{equation}
    \| \blb (\bA^{(k_0)}) \|
    \le
    \sqrt{2 \ebar}. \label{k0_minf}
  \end{equation}

  \noindent Because the vector $\lAz$ is in $\cL$ (defined by (\ref{realisable})), we have
  \begin{equation*}
    \frac{1}{N} \sum_{k=1}^N \|\spcE - \lAk \|^2
    \leq
    \frac{1}{N}\sum_{k=1}^N \| \lAz- \lAk \|^2.
  \end{equation*}
  Expanding the norms squared on both sides yields
  \begin{align}
    \|\spcE\|^2 & - 2  \langle \spcE, \espc \rangle 
    + \frac{1}{N} \sum_{k=1}^N \|\lAk\|^2 \notag \\
    \leq
    & \| \lAz \|^2 - 2  \langle \lAz, \espc \rangle
    + \frac{1}{N} \sum_{k=1}^N \|\lAk\|^2.
  \end{align}
  Subtracting $\frac{1}{N} \sum_{k=1}^N \|\lAk\|^2$ and adding $\big \| \espc \big \|^2$ on both sides we get
  \begin{equation*}
    \big \| \spcE - \espc \big \|^2 \leq
    \big \| \lAz - \espc \big \|^2,
  \end{equation*}
  and therefore
  \begin{equation}
    \| \spcE \|
    \leq 
    \| \lAz \| + 2 \big \| \espc \big \|.
  \end{equation}
  Finally, using lemma~\ref{mean-spectrum-mean-edges} and (\ref{k0_minf}) in the equation above, we obtain
  \begin{equation}
    \| \spcE \|
    \le
    3 \sqrt{2 \ebar},
  \end{equation}
  which completes the proof of the bound on the spectrum of the \fr mean. \qed
\end{lproof}
\subsection{The median graphs computed using the adjacency spectral pseudometric
  \label{median-spectral}}
We finally consider the computation of the median graphs. We have the following bound on the norm of the spectrum of
$\mA$.
\begin{lemma}
  \label{spectral-median}
  Let $\spcM$ be the spectrum of a sample \fr median graph. Let $\ebar$ be the sample mean number of edges of the
  graphs $G^{(1)},\ldots, G^{(N)}$. Then,
  \begin{equation}
    \|  \spcM \| \leq  3 \sqrt{2 \ebar}.
  \end{equation}
\end{lemma}
\begin{lproof}
The function $\Phi$,
  \begin{align*}
    \Phi: \R^n & \longrightarrow [0,\infty)\\
      \bx & \longmapsto \Phi(\bx) = \big \| \spcM - \bx \big\|
  \end{align*}
  is strictly convex, and therefore
  \begin{equation}
    \Phi \big (\espc \big) = 
    \Phi \left (\frac{1}{N} \sum_{k=1}^N \lAk \right)
    \le
    \frac{1}{N} \sum_{k=1}^N  \Phi \left(\lAk \right).
    \label{Phi_convex}
  \end{equation}
  Now, the right-hand side of (\ref{Phi_convex}) is the \fr function evaluated at one of its minimizers. Thus
  $F_1(\spcM)$, is smaller than $F_1(\lAz)$, where $\bA^{(k_0)}$ is defined in lemma~\ref{centerofmass}, and
  (\ref{Phi_convex}) becomes
  \begin{equation}
    \| \spcM  - \espc \| 
    \leq 
    \frac{1}{N}\sum_{k=1}^N \| \lAz - \lAk \|. \label{A41}
  \end{equation}
  Also, because of lemma~\ref{mean-spectrum-mean-edges} and (\ref{k0_minf}), we get
  \begin{equation}
    \frac{1}{N}\sum_{k=1}^N \| \lAz - \lAk \|
    \leq \|\lAz\| + \sqrt{2 \ebar}
    \leq  2 \sqrt{2 \ebar}.\label{A43}
  \end{equation}
  Combining (\ref{A41}) and (\ref{A43}), and using  lemma~\ref{mean-spectrum-mean-edges} we conclude that 
  \begin{equation*}
    \big \| \spcM \big \|
    \leq
    \big \| \espc \big \| + 2 \sqrt{2 \ebar} 
    \leq
    3 \sqrt{2 \ebar}.
  \end{equation*}
  This completes the proof of the bound on the spectrum of a \fr median. \qed
\end{lproof}
\subsection{The number of edges of $\mG$ and $\EG$ when $d = d_\lambda$}
The following lemma provides the bounds given by Theorem~\ref{theorem1} when $d$ is the spectral adjacency pseudometric.
\begin{lemma}
  \label{frechet-edge-spectral}
  Let $\big\{ G^{(1)},\ldots, G^{(N)} \big\}$ be a sample of unweighted simple labeled graphs with vertex set $\left\{1,
  \ldots ,n\right\}$. We consider a sample \fr mean, $\EA$, and a sample \fr median, $\mA$, computed according to the
  spectral adjacency pseudometric. Then
  \begin{equation}
    \max \left\{ \e{\EA} , \e{\mA} \right\} \leq  9  \; \ebar,
  \end{equation}
  where $\ebar$ is the sample mean number of edges given by (\ref{ebar}).
\end{lemma}
\begin{lproof}
  We first analyse the case of a sample \fr mean graph; a sample \fr median graph is handled in the same way. From
  lemmata~\ref{spectral-mean} and~\ref{spectral-median}, we have
  \begin{equation}
    \|  \spcE \|^2  \leq  18 \; \ebar.
  \end{equation}
  Now, from (\ref{edges-eigenval}) we have
  $\e{\EA} = \frac{1}{2} \| \spcE \|^2$,
  and therefore
  \begin{equation*}
    \e{\EA}  \leq 9 \; \ebar,
  \end{equation*}
  which completes the proof of the lemma. \qed
\end{lproof}

\end{document}